\documentclass[12pt, oneside, a4paper]{article}

\usepackage{amsmath}
\usepackage{amsfonts}
\usepackage{amssymb}
\usepackage{amsthm,mathrsfs}
\usepackage{enumerate}
\usepackage{graphicx}
\newtheorem{theorem}{Theorem}[section]
\newtheorem{lemma}[theorem]{Lemma}

\theoremstyle{definition}

\title{\textbf{Groups in which the co-degrees of the irreducible characters are distinct}}
\author{Mahdi Ebrahimi\footnote{ m.ebrahimi.math@ipm.ir}
 \\
 {\small\em  School of Mathematics, Institute for Research in Fundamental Sciences (IPM)},\\{\small\em P.O. Box: 19395--5746, Tehran, Iran}}

\date{}

\begin{document}

\maketitle

\begin{abstract}
 Let $G$ be a finite group and let $\rm{Irr}(G)$ be the set of all irreducible complex characters of $G$. For a character $\chi \in \rm{Irr}(G)$, the number $\rm{cod}(\chi):=|G:\rm{ker}\chi|/\chi(1)$ is called the co-degree of $\chi$. The set of co-degrees of all irreducible characters of $G$ is denoted by $\rm{cod}(G)$. In this paper, we show that for a non-trivial finite group $G$, $|\rm{Irr}(G)|=|\rm{cod}(G)|$ if and only if $G$ is isomorphic to the cyclic group $\mathbb{Z}_2$ or the symmetric group $S_3$.
  \end{abstract}
\noindent {\bf{Keywords:}}  Co-degree, Irreducible character, Camina pair. \\
\noindent {\bf AMS Subject Classification Number:}  20C15, 05C25.

\section{Introduction}
$\noindent$ Let $G$ be a finite group. Also let ${\rm cd}(G)$ be the set of all character degrees of $G$, that is,
 ${\rm cd}(G)=\{\chi(1)|\;\chi \in {\rm Irr}(G)\} $, where ${\rm Irr}(G)$ is the set of all complex irreducible characters of $G$.
  For a character $\chi$ of $G$, the co-degree of $\chi$ is defined as $\rm{cod}(\chi):=|G:\rm{ker}\chi|/\chi(1)$  see \cite{Qian}.
   Clearly, if $\chi$ is irreducible, then $\rm{cod}(\chi)$ is an integer divisor of $|G/\rm{ker} \chi|$.
   Set ${\rm cod}(G):=\{{\rm cod}(\chi)|\; \chi \in {\rm Irr}(G)\}$. It is well known that the
 co-degree set ${\rm cod}(G)$ may be used to provide information on the structure of the group $G$.
  For example, Gagola and Lewis \cite{Gagola} showed that $G$ is nilpotent if and only if $\chi(1)$ divides $\rm{cod}(\chi)$, for every $\chi \in {\rm Irr}(G)$. As another example, Isaacs in \cite{Isaacs} showed that if $G$ is a finite group and $g\in G$, then there exists $\chi \in {\rm Irr}(G)$ so that every prime divisor of $o(g)$ divides $\rm{cod}(\chi)$.

Let $\rm{Irr_1}(G)$ be the set of non-linear characters in $\rm{Irr}(G)$ and $\rm{cd}_1(G)$ be the set of degrees of the characters in $\rm{Irr}_1(G)$. If for a finite group $G$, there exists a non-negative integer $n$ such that $|\rm{cd}_1(G)|=|\rm{Irr}_1(G)|-n$, then $G$ is called a $D_n$-group. In \cite{Bercovich3}, Berkovich et al. gave the classification of $D_0$-groups. In \cite{Bercovich2} and \cite{Bercovich5}, Berkovich and Kazarin classified all $D_1$-groups. In definition of $D_n$-groups,  we may replace the condition $|\rm{cd}_1(G)|=|\rm{Irr}_1(G)|-n$ with $|\rm{cod}(G)|=|\rm{Irr}(G)|-n$ and obtain a new family of finite groups. If for some non-negative integer $n$, $|\rm{cod}(G)|=|\rm{Irr}(G)|-n$, then we say that $G$ is a $D^\prime_n$-group. In this paper, we wish to classify $D^\prime_0$-groups.\\

\noindent \textbf{Main Theorem.}  \textit{Suppose that $G$ is a non-trivial finite group. Then $G$ is a $D^\prime_0$-group if and only if $G$ is isomorphic to $\mathbb{Z}_2$ or $S_3$.}

\section{Preliminaries}
In this section, we state some relevant results on finite groups
needed to prove our main result. We begin with the classification of $D_0$-groups.

\begin{lemma}\label{class}\cite{Bercovich3}.
Let $G$ be a $D_0$-group. Then one of the following cases occurs:\\
\textbf{a)} $G$ is an extra-special $2$-group.\\
\textbf{b)} $G$ is a Frobenius group of order $p^n(p^n-1)$ for some prime power $p^n$ with an abelian Frobenius kernel of order $p^n$ and a cyclic Frobenius complement.\\
\textbf{c)} $G$ is a Frobenius group of order $72$ in which the Frobenius complement is isomorphic to the quaternion group of order $8$.
\end{lemma}

Now we present the classification of rational Frobenius groups due to Darafsheh and sharifi.
\begin{lemma}\label{frobenius}\cite{Darafsheh}.
Let $G$ be a rational Frobenius group. Then $G$ is isomorphic to one of the groups $E(3^n):\mathbb{Z}_2$, $E(3^{2n}):Q_8$ or $E(5^2):Q_8$, where $E(p^n)$ is an elementary abelian $p$-group of order $p^n$. When $G\cong E(3^n):\mathbb{Z}_2$, then $\mathbb{Z}_2$ acts on $E(3^n)$ by inverting every non-identity element.
\end{lemma}

\subsection{Camina pairs}
$\noindent$ Let $N\neq 1$ be a normal subgroup of the finite group $G$. We say the pair $(G,N)$ is a Camina pair if it satisfies the following hypothesis: \\
(F2) If $x \in G-N$, $x$ is conjugate to $xy$ for all $y\in N$. \\
Camina pairs were first introduced by Camina in \cite{Camina}. For the definition of these pairs, he used in \cite{Camina} a character-theoretic approach (see hypothesis (F1) in \cite{Camina}), and proved that his definition is equivalent to hypothesis (F2) above.

\begin{lemma}\cite{Chillag1}.\label{camina}
Let $G$ be a finite group and let $N$ be a non-trivial normal subgroup of $G$. Then $(G,N)$ is a Camina pair if and only if for every $x\in G-N$, we have $|C_G(x)|=|C_{G/N}(xN)|$.
\end{lemma}

Assume that $(G,N)$ is a Camina pair such that $G$ is not a Frobenius group with the Frobenius kernel $N$. The pair $(G,N)$ is called an $F2(p)$-pair if either $N$ or $G/N$ is a $p$-group, for some prime $p$. We refer to \cite{Chillag1} and \cite{Chillag2} for a thorough analysis of this and related topics.

\begin{lemma}\label{abelian}
Assume that for some prime $p$, $(G,N)$ is an $F2(p)$-pair.\\
\textbf{a)} \cite{Chillag1}. If $P$ is a Sylow $p$-subgroup of $G$, then $P$ is non-abelian. \\
\textbf{b)} \cite{Chillag2}. If $G/N$ is a $p$-group, then $G$ has a normal $p$-complement.
\end{lemma}


\section{Proof of Main Theorem}
$\noindent$ In this section, we wish to prove our main result.

\begin{lemma}\label{basic}
Let $G$ be a non-trivial $D'_0$-group. Then\\
\textbf{a)} For every normal subgroup $N$ of $G$, $G/N$ is a $D'_0$-group.\\
\textbf{b)} If $\chi \in \rm{Irr}(G)$, then $\chi$ is rational valued.\\
\textbf{c)} $|G:G'|=2$.\\
\textbf{d)} If $\chi \in \rm{Irr}(G)-\{1_G\}$, then $\rm{ker} \chi \subseteq G'$.\\
\textbf{e)} $(G,G')$ is a Camina pair.
\end{lemma}

\begin{proof}
(a) Since $\rm{cod}(G/N)\subseteq \rm{cod}(G)$ and $\rm{Irr}(G/N)\subseteq \rm{Irr}(G)$, we have nothing to prove.\\
(b) Let $\epsilon$ be a $|G|$th root of unity and $\sigma \in \rm{Gal}(\mathbb{Q}(\epsilon)/\mathbb{Q})$. Then $\chi^\sigma \in \rm{Irr}(G)$ and $\chi^\sigma (1)=\chi(1)$. It is easy to see that $\rm{ker}\chi^\sigma=\rm{ker}\chi$. Thus $\rm{cod}(\chi^\sigma)=\rm{cod}(\chi)$. As $G$ is a $D'_0$-group, $\chi^\sigma=\chi$ and so $\chi^\sigma (x)=(\chi(x))^\sigma=\chi(x)$ for all $x\in G$. Hence $\chi(x)$ is rational for all $x\in G$.\\
(c) Assume that $G=G'$. Let $N$ be a maximal normal subgroup of $G$. By (a), $G/N$ is a non-abelian simple $D'_0$-group. Hence $G/N$ is a non-abelian simple $D_0$-group and so Lemma \ref{class} leads us to a contradiction. Thus $G'<G$. It is easy to see that for every $\lambda \in \rm{Irr}(G/G')$, $\rm{cod}(\lambda)=o(\lambda)$, where $o(\lambda)$ is the order of $\lambda$ in $\rm{Irr}(G/G')\cong G/G'$.  Let $p$ be a prime divisor of $|G/G'|$. There exist $p-1$ distinct irreducible characters $\lambda_1, \lambda_2,\dots, \lambda_{p-1}$ in $\rm{Irr}(G/G')$ so that for every integer $1\leq i\leq p-1$, $\rm{cod}(\lambda_i)=p$. Hence as $G/G'$ is a $D'_0$-group, $p=2$. Therefore $G/G'$ is a $2$-group. If for some $\lambda \in \rm{Irr}(G/G')$, $m:=o(\lambda)>2$, then $\lambda$ and $\lambda^{m-1}$ are distinct irreducible characters in $\rm{Irr}(G/G')$ with same co-degree. It is a contradiction as $G/G'$ is a $D'_0$-group. Thus $G/G'$ is an elementary abelian $2$-group and so $\rm{cod}(G/G')=\{1,2\}$. Hence as $G/G'$ is a $D'_0$-group, we deduce that $|G:G'|=2$.\\
(d) If $\chi$ is linear, then as $|G:G'|=2$, we have $\rm{ker} \chi=G'$. Thus we may assume that $\chi$ is non-linear. By (a), $G/\rm{ker}\chi$ is a $D'_0$-group. Hence using (c), $|G:\rm{ker}\chi G'|=2=|G:G'|$ and so $\rm{ker} \chi\subseteq G'$.\\
(e) Since $|G:G'|=2$, $G$ has a unique non-principal linear character $\lambda$. Let $\chi$ be a non-linear irreducible character of $G$.
 Then $\lambda \chi \in \rm{Irr}(G)$ and $(\lambda \chi)(1)=\chi(1)$.
 By (d), $\rm{ker}\chi$ and $\rm{ker}\lambda\chi$ are subgroups of $G'$. Thus we deduce that $\rm{ker}\lambda\chi=\rm{ker}\chi$ and so $\rm{cod}(\lambda \chi)=\rm{cod}(\chi)$. Thus as $G$ is $D'_0$-group, $\lambda \chi=\chi$.
    Therefore $\chi$ vanishes on $G-\rm{ker}\lambda=G-G'$. Now let $g\in G-G'$.
     As $G/G'$ is abelian, we have $C_{G/G'}(gG')=G/G'$.
    On the other hand $\chi(g)=0$ for all non-linear $\chi\in \rm{Irr}(G)$, and consequently
$|C_G(g)|=1+|\lambda(g)|^2=|G:G'|$.
Thus using Lemma \ref{camina}, $(G,G')$ is a Camina pair.
\end{proof}

\begin{lemma}\label{cyclic}
Let $G$ be a non-trivial $p$-group, for some prime $p$. If $G$ is a $D'_0$-group, then $G\cong \mathbb{Z}_2$.
\end{lemma}

\begin{proof}
 Using Lemma \ref{basic} (c), it suffices to show that  $G$ is abelian. On the contrary, suppose $G$ is non-abelian with minimal possible order. Since $G$ is a $p$-group, $Z(G)\neq 1$. Thus by Lemma \ref{basic} (a) and this fact that $G$ has minimal possible order, $G/Z(G)$ is an abelian $D'_0$-group. Thus using Lemma \ref{basic} (c), $G/Z(G)\cong \mathbb{Z}_2$. It is a contradiction. Hence $G$ is abelian.
\end{proof}

\begin{lemma}\label{symetric}
Let $G$ be a Frobenius $D'_0$-group with the Frobenius kernel $G'$. Then $G\cong S_3$.
\end{lemma}

\begin{proof}
 Using  (c) of Lemma \ref{basic}, $|G:G'|=2$. Let $x\in G$ be an element of order $2$. By Lemma \ref{basic} (b), $G$ is rational. Hence by lemma  \ref{frobenius},  $G\cong G':<x>$, where for some integer $n\geqslant 1$, $G'\cong E(3^n)$. Now let $\lambda_1,\lambda_2 \in  \rm{Irr}(G')-\{1_{G'}\}$. Then for every, $i\in\{1,2\}$, $\chi_i:=\lambda_i^G \in \rm{Irr}(G)$ and using lemma \ref{frobenius}, it is easy to see that $\rm{cod}(\chi_1)=\rm{cod}(\chi_2)$. Hence as $G$ is a $D'_0$-group, $n=1$ and $G\cong S_3$.
\end{proof}

\noindent \textit{Proof of Main Theorem.} Suppose that $G$ is a non-trivial $D'_0$-group. By Lemma \ref{basic} (e), $(G,G')$ is a Camina pair. If for some prime $p$, $G$ is a $p$-group, then by Lemma \ref{cyclic}, $G\cong \mathbb{Z}_2$. Also if $G$ is a Frobenius group with the Frobenius kernel $G'$, then Lemma \ref{symetric} implies that $G\cong S_3$. Thus we may assume that $G$ is neither a $p$-group nor a Frobenius group with the Frobenius kernel $G'$. Using Lemma \ref{basic} (c), $|G:G'|=2$. Hence $(G,G')$ is an $F2(2)$-pair. Let $P$ be a Sylow $2$-subgroup of $G$. By Lemma \ref{abelian} (a), $P$ is non-abelian. Also using Lemma \ref{abelian} (b), $G$ has a normal $2$-complement, say $N$. Thus by Lemma \ref{basic} (a), $G/N\cong P$ is a $D'_0$-group. Hence using Lemma \ref{cyclic}, $P\cong \mathbb{Z}_2$ which is impossible. Conversely, assume that $G$ is isomorphic to $\mathbb{Z}_2$ or $S_3$. Then it is easy to see that $G$ is a $D'_0$-group and it completes the proof.\qed


\section*{Acknowledgements}
This research was supported in part
by a grant  from School of Mathematics, Institute for Research in Fundamental Sciences (IPM).


\end{document}